\documentclass[10pt,reqno]{amsart}
\usepackage{times}
\usepackage[T1]{fontenc}
\usepackage{amssymb}
\usepackage{amsmath,amsthm}
\usepackage{amsfonts}
\usepackage{leftidx}
\usepackage{mathrsfs}
\usepackage{enumerate}	
\usepackage{abstract}
\usepackage{stmaryrd}
\usepackage{ulem}
\usepackage{graphicx}
 \usepackage{hyperref}
\usepackage[top=0.8in, bottom=0.7in, left=0.8in, right=0.8in]{geometry}

\def\R{\mathbb{R}}

\def\p{\partial}

\def\be{\begin{equation}}
\def\ee{\end{equation}}

\newtheorem{theorem}{Theorem}[section]

\newtheorem{lemma}{Lemma}[section]
\newtheorem{proposition}{Proposition}[section]
\theoremstyle{definition}

\theoremstyle{remark}
\newtheorem{remark}{Remark}[section]

\numberwithin{equation}{section}
\begin{document}
 \title[3D   Vlasov-Poisson]{Decay estimates for the $3D$ relativistic and non-relativistic Vlasov-Poisson systems}
\author{Xuecheng Wang}
\address{YMSC, Tsinghua  University, Beijing, China,  10084
}
\email{xuecheng@tsinghua.edu.cn,\quad xuecheng.wang.work@gmail.com 
}

\thanks{}

\maketitle

\begin{abstract}
We study  the small data global regularity problem of the $3D$ Vlasov-Poisson system for both the relativistic case and the non-relativistic case. The main goal of this paper is twofold. (i) Based on a Fourier method, which works systematically for both the relativistic case and the non-relativistic case,  we give a short proof for the  global regularity and the sharp decay estimate for the $3D$ Vlasov-Poisson system. Moreover, we show that the nonlinear solution scatters to a linear solution in both cases.  The result of sharp decay estimates for the non-relativistic case is not new, see   Hwang-Rendall-Vel\'azquez \cite{hwang} and Smulevici \cite{smulevici4}.   (ii) The Fourier method  presented in this paper  serves as    a good comparison   for the study of more complicated    $3D$ relativistic Vlasov-Nordstr\"om system in \cite{wang1} and  $3D$ relativistic Vlasov-Maxwell system in \cite{wang2}.

\end{abstract}

\section{Introduction}

The Vlasov-Poisson system plays a   very important role in plasma physic and galactic dynamics. In three dimensions, it    reads as follows, 
\be\label{vp}
(\textup{VP})\quad \left\{\begin{array}{l}
\displaystyle{\p_t f +  a(v)\cdot \nabla_x f + \mu \nabla_x \phi\cdot \nabla_v f=0} \\
\displaystyle{\Delta \phi = \rho(f):= \int_{\R^3} f(t,x,v) d v}, \\ 
\end{array}\right.
\ee
where $f$ denotes the distribution function  of particles, $\rho$ denotes the density function of particles, $\mu\in\{+,-\}$, $a(v)\in \{v, \hat{v} :=v/\sqrt{m^2+|v|^2/c^2}\}$, ``$m$'' denotes the mass of particles, which is normalized to be one,  and ``$c$'' denotes the speed of light.  If we assume that the  speed of light is   infinity, then    $a(v)=v$, which is called the non-relativistic case. If we assume that the speed of light $c$ is finite, which will be  normalized  to be one, the case  $a(v)=\hat{v}=v/\sqrt{1+|v|^2}$ is called the  relativistic case. Physically speaking, it   means that the speed of massive particles   doesn't exceed the speed of light. $\mu=+$ is called the plasma physics case and    $\mu=-$ is called the stellar dynamics case. The sign of $\mu$ will not play much role in the small data problem. However, it does play a fundamental role  in the large data problem, see \cite{lemou2,lemou3,lemou} and references therein for more details.

There is an extensive literature devoted to the study of the Vlasov-Poisson system in different settings. Instead of trying to elaborate it here, we   refer readers to  Andr\'easson \cite{anderson},   Mouhot\cite{mouhot1} and references therein for more comprehensive introduction.  For our interest,    we   list several representative results    for Cauchy problem. A classic result by Lions-Perthame\cite{lions} says that the solution of the $3D$ non-relativistic Vlasov-Poisson system (\ref{vp}) globally exists as long as the initial data has moments in $v$ higher than $3$, see also \cite{Pfaffelmoser}. For the case  $\mu=- $, Lemou-M\'ehats-Rapha\"el \cite{lemou} showed that spherical models, which are steady solutions of the $3D$ non-relativistic Vlasov-Poisson system (\ref{vp}), are orbital stable under small general perturbations.  In  the periodic setting, a celebrated work by Mouhot-Villani \cite{mouhot2} on the nonlinear Landau Damping  showed that  certain  steady solutions of  the $3D$ non-relativistic Vlasov-Poisson system (\ref{vp}), is stable under small perturbation in   Gevrey spaces, see also Bedrossian-Masmoudi-Mouhot \cite{jacob} for   a new and simplified proof. 

Our main interest in this paper concerns the global  regularity, sharp decay estimates, and the asymptotic behavior of the $3D$ Vlasov-Poisson system (\ref{vp}) for both the non-relativistic case and the relativistic case.  A classic result of  Bardos-Degond \cite{bardos} showed that, for small localized initial data, classic solution   of the $3D$ non-relativistic Vlasov-Poisson system (\ref{vp}) exists globally and  $L^\infty_x$-norms  of    $\nabla_x\phi$ and   $\rho(f)$ decay sharply.

An  interesting question one can ask is whether the derivatives of  density function also decay sharply, which is not answered in \cite{bardos}. Motivated from the linear solution,  one might expect that the following estimate also holds  for the nonlinear solution of the  $3D$  Vlasov-Poisson system,
\be\label{may9eqn1}
|\nabla_x^k \rho(f)(t)| \lesssim C(f(0)) (1+|t|)^{-3-k}, \quad k\in \mathbb{N}_{+}. 
\ee
 
To answer this question, 
one needs  to   
control  precisely  the high order energy of  ``$f(t,x,v)$'' over time, which  was   less studied in the literature until recently by Hwang-Rendall-Vel\'azquez \cite{hwang} and Smulevici \cite{smulevici4}, by using different methods.  The method in \cite{hwang}   is based on the analysis of the characteristics. In \cite{smulevici4}, based on the classic   Klainerman vector fields method \cite{klainerman1,klainerman2},  an   interesting modified vector field method was introduced, see also \cite{fajman} for more details. 

In this paper, we revisit this small data global regularity problem and  introduce a Fourier   method to control precisely the high order energy and prove sharp decay  estimate for the derivatives  of the density function.  

An interesting feature of  our Fourier  method is that it is   robust for both the relativistic case and the   non-relativistic case.   We will handle these two cases uniformly by freely allowing $a(v)\in\{v, \hat{v}\}$. The sharp decay  estimate for the derivatives  of the density function   in the  relativistic case is new.   Generally speaking,  the    relativistic case can differ dramatically from the  non-relativistic case. For example, given a smooth compact support initial data with negative energy in the case $\mu=-$, the solution of the $3D$ non-relativistic Vlasov-Poisson system exists globally, see the result of Schaeffer  \cite{schaeffer2}. However, the solution of   the $3D$ relativistic Vlasov-Poisson system blows up in finite time, see  the  result of Glassey-Schaffer \cite{Glassey1}.

Our main result is stated as follows.
\begin{theorem}\label{maintheorem}
For some large number $N_0\geq 30$, and a sufficiently small number $\delta\in (0, 10^{-9}]$, if the initial  distribution function $f_0(x,v)$ of particles satisfies the following estimate, 
\be\label{initialcondition}
\sum_{|\alpha|+|\beta|\leq N_0} \|(1+|v|+|x|)^{10N_0-8( |\alpha|+|\beta| )}\nabla_x^\alpha\nabla_v^\beta f_0(x,v)\|_{L^2_{x,v} }\leq \epsilon_0,
\ee
where $\epsilon_0$ is a suitably small number, then the $3D$ Vlasov-Poisson system \textup{(\ref{vp})} poses global solution for both the relativistic case and the non-relativistic case.  Moreover, the following energy estimate and sharp decay estimate hold  for both the relativistic case and the non-relativistic case, 
\be\label{energyestimate1}
\sum_{|\alpha|+|\beta|\leq N_0} \|(1+|v|+|x|)^{10N_0-8( |\alpha|+|\beta| )}\nabla_x^\alpha\nabla_v^\beta\big( f(t,x+t a(v), v)\big)\|_{L^2_{x,v} }\leq (1+|t|)^{\delta} \epsilon_0, 
\ee
\be\label{sharpdecayestimate}
 \sum_{|\alpha|\leq N_0-6} (1+|t|)^{3+|\alpha|} \big\|\nabla_x^\alpha\big(\int_{\R^3} f(t,x,v) d v\big) \big\|_{L^\infty_{x}} + (1+|t|)^{2+|\alpha|} \big\| \nabla_x^\alpha \phi(t,x) \big\|_{L^\infty_{x}} \lesssim \epsilon_0.
\ee
\end{theorem}
\begin{remark}
As a byproduct of the proof of the above theorem, we know that the nonlinear solution scatters to  a linear solution as time goes to infinity. More precisely,   there exists scattering data $\tilde{f}_{\pm \infty}(x,v)$ such that,    $\forall v\in \R^3$,  $\displaystyle{   \int_{\R^3}   f(t,x,v) d x  } \longrightarrow \int_{\R^3}   \tilde{f}_{\pm \infty}(x- ta(v),v) d x   = \int_{\R^3}   \tilde{f}_{\pm \infty}(x,v) d x   $ as $t\longrightarrow \pm \infty$.  
\end{remark}
 
 \begin{remark}
 We remark that the choice of weight function can be improved substantially  for the non-relativistic case. For simplicity, we  don't  pursue this goal here. We choose a high order weight function   in (\ref{initialcondition}) because of the unavoidable loss of weight in ``$v$'' in the decay estimate of the density type function for the relativistic case, e.g., see the decay estimate (\ref{decayrvpphi}) in Lemma \ref{decayestimate}.

 \end{remark}

\section{Preliminary and  Proof of the  main theorem}
\subsection{Preliminary}
 For any two numbers $A$ and $B$, we use  $A\lesssim B$ and $B\gtrsim A$  to denote  $A\leq C B$, where $C$ is an absolute constant. We  fix an even smooth function $\tilde{\psi}:\R \rightarrow [0,1]$, which is supported in $[-3/2,3/2]$ and equals to ``$1$'' in $[-5/4, 5/4]$. For any $k\in \mathbb{Z}$, we define
\[
\psi_{k}(x) := \tilde{\psi}(x/2^k) -\tilde{\psi}(x/2^{k-1}), \quad \psi_{\leq k}(x):= \tilde{\psi}(x/2^k)=\sum_{l\leq k}\psi_{l}(x), \quad \psi_{\geq k}(x):= 1-\psi_{\leq k-1}(x).
\]
Moreover, we use  $P_{k}$, $P_{\leq k}$ and $P_{\geq k}$ to denote the projection operators  by the Fourier multipliers $\psi_{k}(\cdot),$ $\psi_{\leq k}(\cdot)$ and $\psi_{\geq k }(\cdot)$ respectively.

For any $k\in \mathbb{Z}$, we define  the $\mathcal{S}^\infty_k$-norm associated with symbols and the class of symbols $\mathcal{S}^\infty$ as follows, 
\[
\|m(\xi)\|_{\mathcal{S}^\infty_k}:=  \sum_{|\alpha|=0, 1,\cdots, 10}2^{|\alpha| k}\|\mathcal{F}^{-1}[\nabla_\xi^\alpha m(\xi)\psi_k(\xi)]\|_{L^1},\]
\[
\mathcal{S}^\infty:=\{m(\xi):  \| m(\xi)\|_{\mathcal{S}^\infty}:=\sup_{k\in \mathbb{Z}} \| m(\xi)\|_{\mathcal{S}^\infty_k} < \infty \}. 
\]

\subsection{Proof of the  main theorem: a bootstrap argument}

In this subsection, we introduce the set-up of energy estimates and reduce the proof of our main theorem \ref{maintheorem} into the proof of two key propositions, which will be postpone to the next section for the sake of clarity.

Firstly, instead of studying the distribution function $f(t,x,v)$ itself, we  study the profile  of the Vlasov-Poisson system (\ref{vp}) ``$g(t,x,v)$'' defined  as follows, 
\[
g(t,x,v):= f(t,x+a(v)t, v), \quad \Longrightarrow \nabla_v f(t,x,v) =  \nabla_v g(t,x- a(v)t,v) - t\nabla_v a(v)\cdot \nabla_x  g(t,x-  a(v)t,v). 
\]

 As a result, the linear effect comes from the transport operator is balanced out, we can derive the following equation satisfied by the profile of the VP system (\ref{vp}) 
\be\label{profilevp}
\p_t g(t,x,v)=  \nabla_x \phi(t,x+a(v)t)\cdot  (\nabla_v - t\nabla_v a(v)\cdot \nabla_x ) g(t,x ,v). 
\ee
Hence, instead of studying the transport operator ``$\p_t + a(v)\cdot \nabla_v$'' acts on $f$, we study ``$\p_t$'' acts on $g$.    The idea of studying the profile is not new, it has been widely used in the study of nonlinear dispersive equation recently.

Due to the small data regime,  it is trivial to prove that the solution exists up to time one and the solution has the same size as the initial data, which is of size $\epsilon_0$, see (\ref{initialcondition}). For simplicity, we assume that we start from time one. We will prove that the maximal time of existence is infinity, i.e., global existence, by using a bootstrap argument.

Firstly, we  define the high order energy   for the Vlasov-Poisson system as follows, 
\be
E_{\textup{high}}^{ }  (t) := E_{\textup{high}}^{1 } (t)+ E_{\textup{high}}^{ 2} (t),\quad E_{\textup{high}}^{ 1} (t):= \sum_{|\alpha|+|\beta|= N_0} \|\omega_{\beta}^\alpha(x,v) \nabla_x^\alpha \nabla_v^\beta g(t,x,v) \|_{L^2_{x,v}},
\ee
\be\label{highorderenergy}
E_{\textup{high}}^{2 } (t):= \sum_{|\alpha|+|\beta|<  N_0} \|\omega_{\beta}^\alpha(x,v) \nabla_x^\alpha \nabla_v^\beta g(t,x,v) \|_{L^2_{x,v}},\quad \omega_{\beta}^\alpha(x,v):=(1+|v|+|x|)^{ 10N_0-8(|\alpha|+|\beta|)} ,
\ee
where we separate out  $E_{\textup{high}}^{1 } (t)$ to  denote  the strictly top order energy. 

Moreover, we define the following low order energy for the Vlasov-Poisson  system
\be\label{loworderenergydef}
 E_{\textup{low}}^{ } (t):= \sum_{|\alpha| \leq N_0 } \|\widetilde{\omega_\alpha}(v)   \big(\nabla_v^\alpha\widehat{g}(t,0,v) -\nabla_v\cdot g_\alpha(t,v)\big)  \|_{  L^2_{ v}},\quad \widetilde{\omega_\alpha}(v) :=(1+|v|)^{10N_0- 8 |\alpha| -20} ,
 \ee 
 where the correction term $g_\alpha(t,v)$ is defined as follows, 
 \be\label{may5eqn1}
g_\alpha(t,v):= \left\{ 
\begin{array}{ll}
 \displaystyle{ \int_{0}^t \int_{\R^3} \nabla_x\phi(s,x+a(v)s) \nabla_v^\alpha g(s,x,v) d x d s   }  & \textup{when}\, |\alpha|=N_0 \\
 0 & \textup{when}\, |\alpha|< N_0.\\
  \end{array}\right.
 \ee
In later argument (the proof of Lemma \ref{loworderenergy}), we will see that the  correction term $g_\alpha(t,v)$  is necessary due to a potential loss of derivative in ``$v$''.

We make the following bootstrap assumption,
\be\label{bootstrap}
\sup_{t\in[0, T)} |t|^{-\delta}  E_{\textup{high}}^{ 1} (t) + |t|^{-\delta/2}  E_{\textup{high}}^{ 2} (t) +  E_{\textup{low}}^{ } (t)\lesssim \epsilon_1:=\epsilon_0^{5/6},
\ee
where, from the local theory,  $T>1$. 

We emphasize that it is important to show that the low order energy $E_{\textup{low}}^{ } (t)$, which is responsible for the sharp decay estimate of derivatives of density function (see  (\ref{decayrvpphi}) in Lemma \ref{decayestimate}),  doesn't grow over time. This fact also motivates us to make the bootstrap assumption as in (\ref{bootstrap}).

The rest of this paper is devoted to prove the following two propositions.  By assuming the validity of these two propositions, we finish the proof of our main theorem \ref{maintheorem}.

\begin{proposition}\label{highorderenergyestimate}
For any $t_1,t_2\in   [1, T)$, s.t., $t_2\geq t_1$, the following estimates hold for  the increment of high order energy with respect to time  
\[
\big(E_{\textup{high}}^{ 1 } (t_2) \big)^2- \big(E_{\textup{high}}^{1 } (t_1) \big)^2\lesssim \sum_{|\alpha|=N_0} \int_{t_1}^{t_2}   |t|^{-1} E_{\textup{high}}^{  } (t)   \big[  E_{\textup{low}}^{ } (t)  E_{\textup{high}}^{  } (t )  + |t|^{-1/2}   E_{\textup{high}}^{ } (t) 
\]
\be\label{highordervp}
+\log(1+|t|) E_{\textup{high}}^{ 2} (t) \| \widetilde{\omega_\alpha}(v) g_{\alpha}(t,v)\|_{L^2_v}  \big] d t,
\ee
\be\label{highordervp2}
\big(E_{\textup{high}}^{ 2 } (t_2) \big)^2- \big(E_{\textup{high}}^{2 } (t_1) \big)^2\lesssim  \int_{t_1}^{t_2}   |t|^{-1} \big(E_{\textup{high}}^{  2} (t) \big)^2  \big[  E_{\textup{low}}^{ } (t)    + |t|^{-1/2}   E_{\textup{high}}^{ 2} (t)    \big] d t.
\ee
Moreover, the following estimate holds for the weighted $L^2_v$-norm  of the correction term $g_{\alpha}(t,v)$ defined in  \textup{(\ref{may5eqn1})},
\be\label{may7eqn31}
\sum_{|\alpha|=N_0} \|\widetilde{\omega_\alpha}(v) g_{\alpha}(t,v)\|_{L^2_v} \lesssim \int_0^t (1+|s|)^{-2} \big(E_{\textup{high}} (s) \big)^2 d s . 
\ee
 \end{proposition}
\begin{proof}
Postponed to subsection \ref{highprop}. 
\end{proof}

\begin{proposition}\label{loworderenergy}
 For any $t_1,t_2\in   [1, T)$, s.t., $t_2\geq t_1$, the following estimate holds,
\be\label{may4eqn45}
\big| E_{\textup{low}}^{ } (t_2) -  E_{\textup{low}}^{  } (t_1) \big| \lesssim \int_{t_1}^{t_2} (1+|t|)^{-2+\delta} \big( E_{\textup{high}}^{ } (t) \big)^2 d t.
\ee
\end{proposition}
\begin{proof}
Postponed to subsection \ref{lowprop}. 
\end{proof}

\noindent \textit{Proof of Theorem \ref{maintheorem}}:\qquad  The proof  of Theorem \ref{maintheorem} is a standard bootstrap argument. 
From    the estimates (\ref{highordervp}--\ref{may7eqn31})  in Proposition \ref{highorderenergyestimate} and the estimate (\ref{may4eqn45}) in Proposition \ref{loworderenergy}, we have
\[
\sup_{t\in[1, T)} |t|^{-\delta}  E_{\textup{high}}^{ 1} (t)+ |t|^{-\delta/2}  E_{\textup{high}}^{ 2} (t)+  E_{\textup{low}}^{ } (t)\lesssim \epsilon_0 +\sup_{t\in[1,T)} \big[|t|^{-\delta} \big(\int_{1}^{t} \big(|s|^{-1+2\delta}\epsilon_1^3
\]
\be\label{may10eqn11}
 +|s|^{-1+3/2\delta}\log(1+|s|) \epsilon_1^3 d s  \big)^{1/2}+|t|^{-\delta/2} \big(\int_{1}^{t}  |s|^{-1+\delta}  \epsilon_1^3 d s \big)^{1/2} +\int_{1}^{t} |s|^{-2+4\delta} \epsilon_1^2 d s\big]\lesssim \epsilon_0.
\ee
Hence ``$T$'' can be extended to infinity and  finishing the bootstrap argument.

The desired energy estimate (\ref{energyestimate1})  holds from the estimate  (\ref{may10eqn11}) and the  desired estimate (\ref{sharpdecayestimate}) holds straightforwardly from the linear decay estimate (\ref{decayrvpphi}) in Lemma \ref{decayestimate}  and the estimate  (\ref{may10eqn11}).  

\qed

\section{Proof of key propositions}

\subsection{Toolkit: three basic lemmas}

In this subsection, we prove three basic lemmas that will be used in the proof of key propositions \ref{highprop} and \ref{lowprop}. 

The first basic lemma concerns the $L^\infty_x$-decay estimate of density type function, which is  controlled by the energy of distribution function. Despite  our final decay estimate (\ref{decayrvpphi}) in Lemma  \ref{decayestimate} looks complicate, it  is actually very delicate and precise because we separate out the leading role of the  low order energy and the role of derivatives. From the estimate (\ref{decayrvpphi}), we know that extra derivative provides extra $1/t$ decay rate in time, which is consistent with the expectation in (\ref{may9eqn1}). 
 
\begin{lemma}\label{decayestimate}
For any fixed $a(v)\in\{v,\hat{v}\}$,  $x\in \R^3$, $a, t\in \R$, s.t, $|t|\geq 1$, $a> -3$, and any given symbol $m(\xi,v)\in L^\infty_v \mathcal{S}^\infty$, the following decay estimate holds, 
\[
 \big|\int_{\R^3}\int_{\R^3} e^{i x\cdot \xi + i t a(v)  \cdot \xi} m(\xi, v)|\xi|^{a}\widehat{g}(t, \xi, v) d v d\xi   \big|\lesssim \sum_{|\alpha|\leq 5+\lfloor a\rfloor} \big(\sum_{|\beta|\leq 5+\lfloor a\rfloor}   \|\nabla_v^\beta m(\xi,v)\|_{L^\infty_v\mathcal{S}^\infty}\big)
 \]
\be\label{decayrvpphi}
\times  \big[|t|^{-3-a}    \|(1+|v|)^{5+|a|} \nabla_v^\alpha \widehat{g}(t,0,v) \|_{L^1_v } +|t|^{-4-a} \| (1+|v|)^{5+|a|} (1+|x|)\nabla_v^\alpha g(t,x,v)\|_{L^1_x L^1_v}\big].
\ee
\end{lemma}
\begin{proof}
After doing dyadic decomposition for the frequency ``$\xi$'' and separating out the zero frequency, we have the following decomposition, 
\[
\int_{\R^3}\int_{\R^3}  e^{i x\cdot \xi + i t a(v)  \cdot \xi} m(\xi,v)|\xi|^{a}\widehat{g}(t, \xi, v) d v d\xi = \sum_{k\in \mathbb{Z}} I_k,\,\, 
\vspace{-0.5\baselineskip}
\]
where
\vspace{-0.5\baselineskip}
\[
I_k= \int_{\R^3}\int_{\R^3} e^{i x\cdot \xi + i t a(v)  \cdot \xi} m(\xi,v)|\xi|^{a}\widehat{g}(t,  0, v)\psi_k(\xi) d v d\xi
\]
\be\label{may2eqn51}
+ \int_0^1 \int_{\R^3}\int_{\R^3} e^{i x\cdot \xi + i t a(v)  \cdot \xi} m(\xi,v)|\xi|^{a}\xi\cdot\nabla_\xi \widehat{g}(t, s\xi, v)\psi_k(\xi) d v d\xi d s.
\ee
Therefore, the following estimate holds after using the volume of support of ``$\xi$'',
\be\label{may2eqn2}
|I_{k}| \lesssim 2^{(3+a)k}\|m(\xi,v)\|_{L^\infty_v \mathcal{S}^\infty_k}\big[\|\widehat{g}(t,0,v) \|_{L^1_v } + 2^{ k}  \| (1+|x|)  {g}(t,x,v)\|_{L^1_x L^1_v}\big].
\ee
On the other hand, after first doing integration by parts in ``$v$''  ``$5+\lfloor a\rfloor$ ''times for the integrals in (\ref{may2eqn51}) and then using the volume of support of ``$\xi$'', we have 
\[
|I_k| \lesssim \sum_{|\alpha|, |\beta| \leq 5+\lfloor a\rfloor} |t|^{-5-\lfloor a\rfloor}2^{(a-\lfloor a\rfloor)k}\|\nabla_v^\beta m(\xi,v)\|_{L^\infty_v \mathcal{S}^\infty_k}\big[2^{-2k }  \| (1+|v|)^{5+|a|} \nabla_v^\alpha \widehat{g}(t,0,v)\|_{L^1_v} 
\]
\be\label{may2eqn21}
+  2^{-k }   \| (1+|v|)^{5+|a|} (1+|x|)\nabla_v^\alpha g(t,x,v)\|_{L^1_x L^1_v}\big].
\ee
In the above estimate, we used the  fact that the following estimate holds regardless $a(v)\in\{v, \hat{v}\}$,
\[
|\nabla_v a(v)|\gtrsim  ({1+|v|})^{-1},\quad |\nabla_v^2 a(v)|\lesssim ({1+|v|})^{-1}.
\]
Therefore, from the estimates (\ref{may2eqn2}) and (\ref{may2eqn21}), we have
\[
\sum_{k\in\mathbb{Z}} |I_k| \lesssim \sum_{k\in \mathbb{Z}, 2^k\leq |t|^{-1}}  2^{(3+a)k}\|m(\xi,v)\|_{L^\infty_v\mathcal{S}^\infty_k}\big[\|\widehat{g}(t,0,v) \|_{L^1_v } + 2^{ k}  \| (1+|x|)  {g}(t,x,v)\|_{L^1_x L^1_v}\big]
\]
\[
+ \sum_{k\in \mathbb{Z}, 2^k\geq |t|^{-1}}\sum_{|\alpha|, |\beta|\leq 5+\lfloor a\rfloor} \|\nabla_v^\beta m(\xi,v)\|_{L^\infty_v \mathcal{S}^\infty_k}  2^{(a-\lfloor a\rfloor)k}|t|^{-5-\lfloor a\rfloor}\big[2^{-2k }  \| (1+|v|)^{5+|a|} \nabla_v^\alpha \widehat{g}(t,0,v)\|_{L^1_v} 
\]
\[
+  2^{-k }   \| (1+|v|)^{5+|a|} (1+|x|)\nabla_v^\alpha g(t,x,v)\|_{L^1_x L^1_v}\big]  \lesssim  \sum_{|\alpha|, |\beta|\leq 5+\lfloor a\rfloor} \|\nabla_v^\beta m(\xi,v)\|_{L^\infty_v \mathcal{S}^\infty}
\]
\[
\times \big[|t|^{-3-a} \|(1+|v|)^{5+|a|} \nabla_v^\alpha \widehat{g}(t,0,v) \|_{L^1_v }   +|t|^{-4-a} \| (1+|v|)^{5+|a|} (1+|x|)\nabla_v^\alpha g(t,x,v)\|_{L^1_x L^1_v}\big].
\]
Hence finishing the proof of our desired estimate (\ref{decayrvpphi}). 
\end{proof}

Our next two    Lemmas    concerns the $L^2_{x,v}$-estimate of the nonlinearity  in (\ref{profilevp}), which is a basic step of energy estimate.  For convenience, we first define a general bilinear form.  For any fixed $a(v)\in \{v,\hat{v}\}$, fixed sign $\mu\in\{+,-\}$, any two localized distribution functions $f_1(t,x,v)$ and $f_2(t,x,v)$, any fixed $k\in\mathbb{Z}$, any symbol $m(\xi,v)\in L^\infty_v \mathcal{S}^\infty_k$, and  any differentiable coefficient $c(v)$,  we define a bilinear operator as follows, 
\be\label{noveq260}
B_k  (f_1,f_2)(t,x,v):= f_1(t,x,v) E (P_{k}[f_2(t)])(x+a(v)t), 
\ee
where
\[
E(P_k[f])(t,x):=\int_{\R^3}\int_{\R^3} e^{i x \cdot \xi} e^{-i \mu t a(u)\cdot \xi} {c(u)  m(\xi,u) \psi_k(\xi)}  \widehat{f}(t, \xi, u) d\xi d u. 
\]
 From the Poisson equation satisfied by $\phi$ in (\ref{vp}), it's straightforward  to see that the nonlinearity in (\ref{profilevp}) is of the type defined in (\ref{noveq260}) once we localize the frequency of $\nabla_x \phi$.

The study of  bilinear estimates in Lemma \ref{bilineardensitylemma} and Lemma \ref{bilinearestimatelemma2} are motivated from a particular scenario in the high order energy estimate.  Since it's possible that  $\nabla_x\phi$  of the nonlinearity  in (\ref{profilevp}), which corresponds to the input $f_2$ in (\ref{noveq260}),  has the maximal derivatives  in the energy estimate, it means that the input $f_2(t,x,v)$ is not allowed to lose derivatives in the  $L^2_{x,v}$-estimate of the bilinear operator $B_k  (f_1,f_2)(t,x,v)$. It also explains why the input $f_2(t,x,v)$ doesn't lose derivatives  in both estimates (\ref{bilineardensity}) and (\ref{bilineardensitylargek}). Moreover, these non-trivial bilinear estimates are  necessary because the rough $L^2-L^\infty$ type  bilinear estimate for the aforementioned scenario is not sufficient to close the energy estimate  due to the coefficient ``$t$'' in the nonlinearity of the equation satisfied by the profile ``$g(t,x,v)$''  in (\ref{profilevp}),

To prove the bilinear estimates  (\ref{bilineardensity}) and (\ref{bilineardensitylargek}), we rely crucially on  the  analysis of the space resonance set.  The first bilinear estimate (\ref{bilineardensity}) in Lemma \ref{bilineardensitylemma} will mainly be used in the relatively low frequency case, which is also why  we restrict the size of $k$ in the statement.

 \begin{lemma}\label{bilineardensitylemma}
Given any fixed $a(v)\in \{\hat{v}, v\}$,    $k\in \mathbb{Z},$ s.t., $ |t|^{-1}\lesssim  2^{k}\leq 1$, $t\in \R, |t|\geq 1$, and any localized differentiable function $f_3(t,v):\R_t\times \R_v^3\longrightarrow \mathbb{C}$.  
 The  following bilinear estimate  holds,
 \[
\| B_k (f_1,f_2)(t,x  ,v)\|_{L^2_x L^2_v}  \lesssim \sum_{|\alpha|\leq 5}  \big( \| m(\xi, v)\|_{L^\infty_v \mathcal{S}^\infty_k} + \| \nabla_v m(\xi, v)\|_{L^\infty_v \mathcal{S}^\infty_k} \big)  \|(1+ |v|+|x|)^{20}  \nabla_v^\alpha f_1(t,x,v)\|_{L^2_xL^2_v}
     \]
 \[
 \times  \big[ |t|^{-2}2^{k} \|\big(|c(v)|+|\nabla_v c(v)|\big) f_3(t,v) \|_{L^2_v} + |t|^{-3}  2^{ k }    \|(1+ |v|+|x|)^{20}c(v) f_2(t,x,v) \|_{  L^2_x L^2_v} 
 \]
 \be\label{bilineardensity}
 +  |t|^{-3} \| c(v)\big( \widehat{f_2}(t,0,v ) -\nabla_v\cdot f_3(t, v)\big)\|_{L^2_v}   \big], \quad 
 \ee
 \end{lemma}
 
	\begin{proof}

	Recall  the detailed formula of ``$B_k  (f_1,f_2)(t,x,v)$'' in (\ref{noveq260}). On the Fourier side, we have
	\[
	\mathcal{F}[B_k (f_1,f_2)](t,\xi, v)= \int_{\R^3} \widehat{f_1}(t, \xi-\eta, v)  \int_{\R^3}  e^{ita(v)\cdot \eta - i t \mu a(u)\cdot \eta} \psi_k(\eta)   {c(u)  m(\eta,u)  }  \widehat{f_2}(t, \eta, u)  d u d\eta.
	\]
Note that the following estimate holds for any $a(v)\in \{v, \hat{v}\}$,
\be\label{may26eqn110}
 \big|\nabla_\eta\big( a(v)\cdot \eta -  \mu a(u)\cdot \eta \big)\big|  \gtrsim |v-\mu u|(1+|v|)^{-2} (1+|u|)^{-2}. 
\ee
 Motivated from  the size of space resonance set, we choose a rough threshold of $|v-\mu u|$  to be   ``$t^{-1}2^{-k}$''. 

	More precisely,  after separating out the zero frequency of $f_2(t,x,v)$ on the Fourier side, separating out the correction term ``$\nabla_u\cdot f_3(t,u)$'', doing integration by parts in $u$ for the correction term and decomposing into two parts based on the size of ``$v-\mu u$'', the following decomposition holds, 
\be\label{may3eqn96}
\mathcal{F}[B_k (f_1,f_2)](t,\xi, v)=\sum_{i=1,2,3,j=1,2} I_i^j(t, \xi, v),
\ee
where
\[
 I_1^j(t,\xi, v):=\int_{\R^3} \widehat{f_1}(t, \xi-\eta, v)  \int_{\R^3}  e^{ita(v)\cdot \eta - i t \mu a(u)\cdot \eta} \psi_k(\eta)   \big[   m(\eta,u) c(u)\big( \widehat{f_2}(t,0, u ) -\nabla_u\cdot f_3(t,u)\big)
\]
\be\label{may3eqn41}
  - \nabla_u \big( m(\eta,u)c(u)\big)\cdot f_3(t,u) \big]\phi_j(v,u) d u    d \eta, 
\ee
\be\label{may26eqn51}
 I_2^j (t,\xi, v):=\int_{\R^3} \widehat{f_1}(t, \xi-\eta, v)  \int_{\R^3}  e^{ita(v)\cdot \eta - i t \mu a(u)\cdot \eta}   i t\mu  \nabla_u\big( a(u)\cdot \eta\big) \cdot f_3(t,u)  m(\eta,u)c(u) \psi_k(\eta) \phi_j(v,u)   d u   d \eta,
 \ee
\be\label{may3eqn42}
 I_3^j(t,\xi, v):= \int_{0}^1 \int_{\R^3} \widehat{f_1}(t, \xi-\eta, v)  \int_{\R^3}  e^{ita(v)\cdot \eta - i t \mu a(u)\cdot \eta}m(\eta,u) \eta \cdot\nabla_\eta \widehat{f_2}(t,s\eta, u ) c(u)\psi_k(\eta)\phi_j(u,v)  d u   d \eta d s, 
  \ee
where
\be\label{cutofffunction}
\phi_1 (v,u)  = \psi_{\leq 10}(2^{k}|t|  (v-\mu u)) , \quad \phi_2(v,u)= \psi_{ > 10}(2^{k}|t|  (v-\mu u)). 
\ee

\noindent $\oplus$\quad The case when $|v-\mu u|\leq |t|^{-1}2^{-k}$, i.e., the estimate of $I_i^1, i\in \{1,2,3\}.$

For this case, we mainly exploit the smallness of the space resonance set.   Note that,  in terms of kernel, the following equalities holds 
\[
\mathcal{F}^{-1}[I_1^1 ](t,x-ta(v),v)= f_1(t,x-ta(v),v) \int_{\R^3} K_k^1(t,x-t\mu a(u),v, u) \big[ c(u)\big( \widehat{f_2}(t,0, u ) -\nabla_u\cdot f_3(t,u)\big) d u 
\]
\be\label{may25eqn31}
+\nabla_u c(u)\cdot f_3(t,u) \big] +  c(u) K_k^2(t,x-t\mu a(u),v, u) \cdot f_3(t,u) d u, 
\ee
\be\label{may25eqn33}
\mathcal{F}^{-1}[I_2^1  ](t,x-ta(v),v)= f_1(t,x-ta(v),v) \int_{\R^3} c(u) K_k^3(t,x-t\mu a(u),v, u)\cdot f_3(t, u) d u, 
\ee
\be\label{may26eqn1}
\mathcal{F}^{-3}[I_3^1 ](t,x-ta(v),v)= f_1(t,x-ta(v),v) \int_0^1 \int_{\R^3} \int_{\R^3}K_k^4(t,y,v, u)\cdot P_{[k-4,k+4]}[\tilde{f_2}](t,s,x-t\mu a(u)-y  , u ) d y d u d s, 
\ee
where
\vspace{-0.5\baselineskip}
\[
K_k^1(t,x , v,u):=\int_{\R^3} e^{i x\cdot\xi } m(\xi, u) \psi_k(\xi) \psi_{\leq 10}(2^{k}|t|( v - \mu u) ) d \xi,
\]
\[
K_k^2(t,x , v,u):=\int_{\R^3} e^{i x\cdot\xi } \nabla_u m(\xi, u) \psi_k(\xi) \psi_{\leq 10}(2^{k}|t|( v - \mu u) )  d \xi,
\]
\[
K_k^3(t,x , v,u):=\int_{\R^3} e^{i x\cdot\xi }  i t\mu  \nabla_u\big( a(u)\cdot \xi\big) m(\xi, u) \psi_k(\xi) \psi_{\leq 10}(2^{k}|t|( v - \mu u) )  d \xi,
\]
\[
K_k^4(t,x , v,u):=\int_{\R^3} e^{i x\cdot\xi } m(\xi, u)\xi  \psi_k(\xi) \psi_{\leq 10}(2^{k}|t|( v - \mu u) )  d \xi, 
\]
\be\label{may26eqn3}
P_{[k-4,k+4]}[\tilde{f_2}](t,s,x,u):= \int_{\R^3} e^{ix \cdot \xi} c(u) \nabla_\xi \widehat{f_2}(t, s\xi, u)\psi_{[k-4,k+4]}(\xi) d \xi. 
\ee
After doing integration by parts in $\xi$ five times, the following pointwise estimate holds for the kernels, 
\[
 | K_k^1(t,x,v, u)|+ | K_k^2(t,x,v, u)| + (|t|2^{k})^{-1} | K_k^3(t,x,v, u)|+ 2^{-k}| K_k^4 (t,x,v, u)|
\]
\be\label{estimateofkernels1}
   \lesssim 2^{3k}(1+2^k|x|)^{-5}\psi_{\leq 10}(2^{k}|t|( v - \mu u) )\big( \| m(\xi, v)\|_{L^\infty_v \mathcal{S}^\infty_k} + \| \nabla_v m(\xi, v)\|_{L^\infty_v \mathcal{S}^\infty_k} \big)  .
\ee
Therefore, from the  duality of $L^2_{x,v}$  and the estimate of kernels in (\ref{estimateofkernels1}), we have
 \[
 \|\mathcal{F}^{-1}[I_2^1 ](t,x  	,v)\|_{L^2_x L^2_v} \lesssim \sup_{  \|g\|_{L^2_xL^2_v}\leq 1}\int_{\R^3} \int_{\R^3} \int_{\R^3}|t|2^{4k} \big( \| m(\xi, v)\|_{L^\infty_v \mathcal{S}^\infty_k} + \| \nabla_v m(\xi, v)\|_{L^\infty_v \mathcal{S}^\infty_k} \big)   \]
 \[
 \times |g(x,v)| |f_1(t,x-ta(v),v)| |c(u) f_3(u)|  \psi_{\leq 10}(2^{k}|t|( v - \mu u) ) d u d x d v 
 \]
\[
\lesssim  \sup_{  \|g\|_{L^2_x L^2_v}\leq 1}\int_{\R^3} \int_{\R^3} \int_{\R^3}|t|2^{4k} \big( \| m(\xi, v)\|_{L^\infty_v \mathcal{S}^\infty_k} + \| \nabla_v m(\xi, v)\|_{L^\infty_v \mathcal{S}^\infty_k} \big)     |g(x, \mu u + w)| 
\]
\[ 
\times    |f_1(t,x-ta(\mu u + w),\mu u + w )| |c(u) f_3(u)| \psi_{\leq 10}(2^k|t| w ) d x  d u d w 
\] 
\be\label{may26eqn101}
\lesssim |t|^{-2} 2^{  k  } \big( \| m(\xi, v)\|_{L^\infty_v \mathcal{S}^\infty_k} + \| \nabla_v m(\xi, v)\|_{L^\infty_v \mathcal{S}^\infty_k} \big)   \|f_1(t,x,v)\|_{L^\infty_vL^2_x} \| c(u)f_3(u)\|_{L^2_u}.
\ee
 With minor modifications, the following estimate holds for $I_{1}^1(t,\xi,v)$,
\[
 \|\mathcal{F}^{-1}[I_1^1 ](t,x  	,v)\|_{L^2_x L^2_v} \lesssim |t|^{-3}  \big( \| m(\xi, v)\|_{L^\infty_v \mathcal{S}^\infty_k} + \| \nabla_v m(\xi, v)\|_{L^\infty_v \mathcal{S}^\infty_k} \big) \|f_1(t,x,v)\|_{L^\infty_vL^2_x}
\]
 \be\label{may26eqn271}
  \times   \big(\| c(u)\big( \widehat{f_2}(t,0, u ) -\nabla_u\cdot f_3(t,u)\|_{L^2_u} +\| (|c(u)|+|\nabla_u c(u)|) f_3(t,u)\|_{L^2_u} \big).
 \ee
Lastly, we estimate $I_3^1(t,\xi,v)$. Recall (\ref{may26eqn1}) and (\ref{may26eqn3}). From the $L^\infty_x L^2_v-L^2_x L^\infty_v$ type bilinear estimate,  the estimate of the  kernel $K_k^4(t,y,v,u)$ in (\ref{estimateofkernels1}), the decay estimate (\ref{decayrvpphi}) in Lemma \ref{decayestimate}, and the volume of support of $\tilde{f_2}(t,s,x,v)$ on the Fourier side, we have
\[
\|\mathcal{F}^{-1}[I_3^1  ](t,x  	,v)\|_{L^2_x L^2_v} \lesssim \| K_k^4(t,y,v,u)\|_{L^1_y L^\infty_v L^2_u} \|P_{[k-4,k+4]}[\tilde{f_2}](t,s,x,u)\|_{L^\infty_s L^2_xL^2_u}\| f_1(t, x-ta(v),v)\|_{L^\infty_x L^2_v}
\]
\[
  \lesssim \sum_{|\alpha|\leq 5} |t|^{-3}2^{ -k/2 }2^{3k/2}  \big( \| m(\xi, v)\|_{L^\infty_v \mathcal{S}^\infty_k} + \| \nabla_v m(\xi, v)\|_{L^\infty_v \mathcal{S}^\infty_k} \big)  \|(1+|x|+|v|)^{10}c(v) f_2(t,x,v)\|_{L^2_{x,v}}\]
\be\label{may26eqn211}
\times   \|(1+ |v|+|x|)^{10}  \nabla_v^\alpha f_1(t,x,v)\|_{L^2_xL^2_v}. 
\ee

\noindent $\oplus$\quad The case when $|v-\mu u|\geq |t|^{-1}2^{-k}$,  i.e., the estimate of $I_i^2, i\in \{1,2,3\}.$

For this case, we mainly exploit the fact that we are away from the space resonance set by doing integration by parts in ``$\eta$''.  Recall the detailed formula of $I_{i}^2(t,\xi,v)$, $i\in\{1,2,3\}$, in  (\ref{may3eqn41}), (\ref{may26eqn51}), and   (\ref{may3eqn42}). For $I_{1}^2(t,\xi,v)$ and $I_{2}^2(t,\xi,v)$, we do integration by parts in ``$\eta$'' four times. For $I_{3}^2(t,\xi,v)$, we do integration by parts in $\eta$ twice. As a result, we have
\[
 I_1^2(t,\xi, v):=\int_{\R^3}\int_{\R^3}e^{ita(v)\cdot \eta - i t \mu a(u)\cdot \eta} t^{-4}\nabla_\eta \cdot\big(  \frac{a(v)-\mu a(u)}{|a(v)-\mu a(u)|^2} \circ\cdots\circ \nabla_\eta \cdot\big(  \frac{a(v)-\mu a(u)}{|a(v)-\mu a(u)|^2}    \widehat{f_1}(t, \xi-\eta, v)     \psi_k(\eta) 
\]
\be\label{may26eqn91}
 \times  \big(  m(\eta,u) c(u)\big( \widehat{f_2}(t,0, u ) -\nabla_u\cdot f_3(t,u)\big) - \nabla_u \big( m(\eta,u)c(u)\big)\cdot f_3(t,u) \big)\big)\circ\cdots\big)\psi_{>10}(2^{k}|t|(v-\mu u)) d u    d \eta, 
\ee
\[
 I_2^2(t,\xi, v):=\int_{\R^3}\int_{\R^3}e^{ita(v)\cdot \eta - i t \mu a(u)\cdot \eta} i t^{-3}\nabla_\eta \cdot\big(  \frac{a(v)-\mu a(u)}{|a(v)-\mu a(u)|^2} \circ\cdots\circ \nabla_\eta \cdot\big(  \frac{a(v)-\mu a(u)}{|a(v)-\mu a(u)|^2}      \widehat{f_1}(t, \xi-\eta, v)
\]
\be\label{may26eqn92}
 \times  \big(  \mu  \nabla_u\big( a(u)\cdot \eta\big) \cdot f_3(t,u)  m(\eta,u) \big)\big)\circ\cdots\big)\psi_{>10}(2^{k}|t|(v-\mu u)) d u    d \eta, 
\ee
\[
 I_3^2(t,\xi, v):= \int_{0}^1 \int_{\R^3} \int_{\R^3}  e^{ita(v)\cdot \eta - i t \mu a(u)\cdot \eta} t^{-2}  \nabla_\eta \cdot\big(  \frac{a(v)-\mu a(u)}{|a(v)-\mu a(u)|^2}  \nabla_\eta \cdot\big(  \frac{a(v)-\mu a(u)}{|a(v)-\mu a(u)|^2}  \widehat{f_1}(t, \xi-\eta, v)  \]
\be\label{may26eqn93}
 \times     m(\eta,u) \eta \cdot\nabla_\eta \widehat{f_2}(t,s\eta, u )  c(u)\psi_k(\eta)\big)\big)\psi_{> 10}(2^{k}|t|(v-\mu u))  d u   d \eta d s.
\ee

Similar to the strategies used in  the estimate of $ I_2^1(t,\xi, v)$ in (\ref{may26eqn101}), the following estimate holds after writing ``$\mathcal{F}^{-1}[I_2^2](t,x-ta(v), v)$'' in terms of kernels and using the estimate (\ref{may26eqn110}), 
\[
\|\mathcal{F}^{-1}[I_2^2](t,x , v)\|_{L^2_x L^2_v} \lesssim \int_{\R^3} |t|^{-3}  |w|^{-4} \psi_{>10}(|t|2^k w)\big( \| m(\xi, v)\|_{L^\infty_v \mathcal{S}^\infty_k} + \| \nabla_v m(\xi, v)\|_{L^\infty_v \mathcal{S}^\infty_k} \big)\]
\[
\times    \|(1+|v|+|x|)^{15} f_1(t,x,v)\|_{L^\infty_vL^2_x} \|(1+|u|)^{10} c(u)f_3(u)\|_{L^2_u} d w\lesssim |t|^{-2}2^{k} \big( \| m(\xi, v)\|_{L^\infty_v \mathcal{S}^\infty_k}\]
\be\label{may26eqn121}
 + \| \nabla_v m(\xi, v)\|_{L^\infty_v \mathcal{S}^\infty_k} \big)  \|(1+|v|+|x|)^{15} f_1(t,x,v)\|_{L^\infty_vL^2_x} \|(1+|u|)^{10} c(u)f_3(u)\|_{L^2_u}.
\ee
In the same spirit, we have 
\[
\|\mathcal{F}^{-1}[I_1^2](t,x , v)\|_{L^2_x L^2_v} \lesssim |t|^{-3}  \big( \| m(\xi, v)\|_{L^\infty_v \mathcal{S}^\infty_k} + \| \nabla_v m(\xi, v)\|_{L^\infty_v \mathcal{S}^\infty_k} \big)  \|(1+|x|+|v|)^{15} f_1(t,x,v)\|_{L^\infty_vL^2_x}
\]
\be\label{may26eqn281}
 \times \|(1+|u|)^{10} c(u)f_3(u)\|_{L^2_u}.
\ee
Lastly, we estimate  $ I_3^2(t,\xi, v)$. Recall (\ref{may26eqn93}). For this term, we use the same strategy used in the estimate of $ I_3^1(t,\xi, v)$ in (\ref{may26eqn211}). After writing $ I_3^2(t,\xi, v)$ in terms of kernel and then use the $L^\infty_x L^2_v-L^2_x L^\infty_v$ type bilinear estimate, the following estimate holds, 
\[ 
\|\mathcal{F}^{-1}[I_3^2](t,x , v)\|_{L^2_x L^2_v} \lesssim \sum_{|\alpha|\leq 5} |t|^{-3/2} |t|^{-2}  2^k (1+2^{-2k})  2^{3k/2}\big( \| m(\xi, v)\|_{L^\infty_v \mathcal{S}^\infty_k} + \| \nabla_v m(\xi, v)\|_{L^\infty_v \mathcal{S}^\infty_k} \big)  \]
\be\label{may26eqn311}
\times \|(1+|x|+|v|)^{20}c(v) f_2(t,x,v)\|_{L^2_{x,v}}  \|(1+ |v|+|x|)^{20}  \nabla_v^\alpha f_1(t,x,v)\|_{L^2_xL^2_v}. 
\ee
We remark that the factor $(1+2^{-2k})$ appears in the estimate (\ref{may26eqn311}) because it is possible that $\nabla_\eta$ hits $\widehat{f_1}(t, \xi-\eta,v)$ or $\widehat{f_2}(t, \eta, u)$ when doing integration by parts in $\eta$.

To sum up, our desired estimate (\ref{bilineardensity}) holds  from  the decomposition (\ref{may3eqn96}), the estimates (\ref{may26eqn101}), (\ref{may26eqn271}),   (\ref{may26eqn211}), (\ref{may26eqn121}), (\ref{may26eqn281}), and (\ref{may26eqn311}) and the $L^\infty_v\rightarrow L^2_v$ type Sobolev embedding.   
\end{proof}

Next, we show another bilinear estimate (\ref{bilineardensitylargek})  that   will mainly be used in the  relatively high frequency case.
\begin{lemma}\label{bilinearestimatelemma2}
For any $k\in \mathbb{Z}$, the following bilinear estimate holds ,
 \[
\| B_k (f_1,f_2)(t,x  ,v)\|_{L^2_x L^2_v}  \lesssim \sum_{|\alpha|\leq 5}\min\{ |t|^{-3}   ,2^{3k}\}   \|m(\xi,v)\|_{L^\infty_v \mathcal{S}^\infty_k}    \|(1+ |v|+|x|)^{20}c(v) f_2(t,x,v) \|_{  L^2_x L^2_v}
  \]
\be\label{bilineardensitylargek}
 \times    \|(1+ |v|+|x|)^{20} \nabla_v^\alpha f_1(t,x,v)\|_{L^2_xL^2_v} .
\ee
\end{lemma}
\begin{proof}
Note that the following estimate holds after using the $L^\infty_{x,v}-L^2_{x,v}$ type bilinear estimate and the volume of support of the frequency of $f_2(t,x,v)$, 
	\be\label{may26eqn901}
\| B_k(f_1,f_2)\|_{L^2_{x,v}} \lesssim  2^{3k}\|m(\xi,v)\|_{L^\infty_v \mathcal{S}^\infty_k}\|f_1(t,x,v)\|_{L^2_{x,v}}\|(1+|x|+|v|)^5 c(v) f_2(t,x,v)\|_{L^2_{x,v}}.
	\ee

	Alternatively,  similar to the decomposition in (\ref{may3eqn96}),  the following decomposition holds after localizing around the space resonance set,  
	\be\label{may26eqn801}
	\mathcal{F}[B_k (f_1,f_2)](t,\xi, v)= J_1(t,\xi, v) + J_2(t,\xi,v),
	\ee
	where
\[
 J_1  (t,\xi, v):=   \int_{\R^3} \widehat{f_1}(t, \xi-\eta, v)  \int_{\R^3}  e^{ita(v)\cdot \eta - i t \mu a(u)\cdot \eta}m(\eta,u)   \widehat{f_2}(t, \eta, u )      c(u)\psi_{k}(\eta)\psi_{\leq 10}(2^{k}|t|(v-\mu u))  d u   d \eta,
\]
\[
 J_2  (t,\xi, v):=   \int_{\R^3} \widehat{f_1}(t, \xi-\eta, v)  \int_{\R^3}  e^{ita(v)\cdot \eta - i t \mu a(u)\cdot \eta}m(\eta,u)   \widehat{f_2}(t, \eta, u )      c(u)\psi_{k}(\eta)\psi_{>  10}(2^{k}|t|(v-\mu u))  d u   d \eta.
	\]

Similar to the strategy we used in the estimate of $\mathcal{F}^{-1}[I_3^1](t,x,v)$ in (\ref{may26eqn211}), the following estimate holds after first writing $\mathcal{F}^{-1}[J_1](t,x,v)$ in terms of kernels and then using the $L^\infty_x L^2_v-L^2_x L^\infty_v$ type bilinear estimate, 
\[
\|\mathcal{F}^{-1}[J_1](t,x,v)\|_{L^2_x L^2_v} \lesssim \sum_{|\alpha|\leq 5} |t|^{-3} 2^{-3k/2}  \| m(\xi, v)\|_{L^\infty_v \mathcal{S}^\infty_k}\| (1+|x|+|v|)^{10} \nabla_v^\alpha f_1(t, x ,v)\|_{L^2_x L^2_v}
\vspace{-0.3\baselineskip}
\]
\[
\times  \| P_{[k-4,k+4]}[f_2](t,x,u)\|_{L^2_{x,u}}\lesssim  \sum_{|\alpha|\leq 5} |t|^{-3} 2^{-3k_{+}/2}  \| m(\xi, v)\|_{L^\infty_v \mathcal{S}^\infty_k}\| (1+|x|+|v|)^{10}\nabla_v^\alpha f_1(t, x ,v)\|_{L^2_x L^2_v}
\vspace{-0.3\baselineskip}
\]
\be\label{may26eqn381}
\times \| (1+|x|+|u|)^{5} f_2 (t,x,u)\|_{L^2_{x,u}}
\ee 

 Similar to the strategy we used in the estimate of $ I_3^2 (t,\xi,v)$ in (\ref{may26eqn311}), we also do integration by parts in ``$\eta$'' twice for $J_2(t,\xi, v)$. After  writing $\mathcal{F}^{-1}[J_2](t,x,v)$ in terms of kernels and then using the $L^\infty_x L^2_v-L^2_x L^\infty_v$ type bilinear estimate,  the following estimate holds, 
\[
\|\mathcal{F}^{-1}[J_1](t,x,v)\|_{L^2_x L^2_v} \lesssim \sum_{|\alpha|\leq 5}   |t|^{-3/2} |t|^{-2}    (1+2^{-2k})  \| m(\xi, v)\|_{L^\infty_v \mathcal{S}^\infty_k}\| (1+|x|+|v|)^{20} \nabla_v^\alpha f_1(t, x ,v)\|_{L^2_x L^2_v} \]
\[
\times \min\{1,2^{3k/2}\} \| (1+|x|+|u|)^{20}  f_2 (t,x,u)\|_{L^2_{x,u}}\lesssim \sum_{|\alpha|\leq 5} |t|^{-7/2} 2^{ - k_{-}/2}    \| m(\xi, v)\|_{L^\infty_v \mathcal{S}^\infty_k}
\]
\be\label{may26eqn610}
\times \| (1+|x|+|v|)^{20} \nabla_v^\alpha f_1(t, x ,v)\|_{L^2_x L^2_v}  \| (1+|x|+|u|)^{20}  f_2 (t,x,u)\|_{L^2_{x,u}}. 
\ee
To sum up,  our desired estimate (\ref{bilineardensitylargek}) holds from the decomposition (\ref{may26eqn801}) and the estimates (\ref{may26eqn901}), (\ref{may26eqn381}) and (\ref{may26eqn610}). 
\end{proof}

\subsection{Proof of Proposition \ref{highorderenergyestimate}}\label{highprop} 

  \begin{proof}
Recall (\ref{vp}) and (\ref{profilevp}). It is easy to see that the following equality  holds, 
\[
 \frac{1}{2}\big(E_{\textup{high}}^{ 1}(t_2)\big)^2- \frac{1}{2} \big(E_{\textup{high}}^{ 1}(t_1)\big)^2 = \sum_{ |\alpha|+|\beta|=  N_0} \sum_{\alpha_1+\alpha_2=\alpha, \beta_1+\beta_2=\beta, |\alpha_2|+|\beta_2|\leq N_0 } J_{\alpha_1,\alpha_2}^{\beta_1,\beta_2},
 \]
where
\[
J_{\alpha_1,\alpha_2}^{\beta_1,\beta_2}= \int_{t_1}^{t_2} \int_{\R^3}  \int_{\R^3} \nabla_v^{\beta_1}\big(\nabla_x^{\alpha_1+1}\phi(t,x+a(v)t)\big)\cdot  (\nabla_{x}^{\alpha_2} \nabla_v^{\beta_2+1}g(t,x ,v) - t C_{\beta_1, \beta_2}^{\alpha_1, \alpha_2}(v) \cdot \nabla_x^{\alpha_2+1}\nabla_v^{\beta_2} g(t,x ,v) ) )
\]
\be\label{april30eqn1}
 \times   \big( \omega_{\beta}^\alpha(x,v)\big)^2 \nabla_x^\alpha \nabla_v^\beta g(t,x,v)    d x d v,
\ee
where $C_{\beta_1, \beta_2}^{\alpha_1, \alpha_2}(v)$ are some uniquely determined coefficients that satisfy the following estimate and equality, 
\[
|C_{\beta_1, \beta_2}^{\alpha_1, \alpha_2}(v)| \lesssim 1, \quad C_{0, \beta }^{0, \alpha }(v)= \nabla_v a(v).
\] 

\vspace{0.4\baselineskip}

\noindent \textbf{Step 1:} Further reduction of $J_{\alpha_1,\alpha_2}^{\beta_1,\beta_2}$.

\vspace{0.4\baselineskip}

For the case when $\alpha_2=\alpha, \beta_2=\beta$, i.e., the case when all the derivatives hit on ``$(\nabla_v-t\nabla_v a(v)\cdot\nabla_x )g(t,x,v)$'',   we use the fact that
\[
\nabla_x^\alpha \nabla_v^\beta g(t,x,v) ( \nabla_v - t\nabla_v a(v)\cdot \nabla_x)\nabla_x^\alpha \nabla_v^\beta g(t,x,v)= \frac{1}{2} ( \nabla_v - t\nabla_v a(v)\cdot \nabla_x)\big( \nabla_x^\alpha \nabla_v^\beta g(t,x,v)\big)^2.
\]
Hence, after doing integration by parts in ``$v$'' and ``$x$''. The following equality holds, 
\[
J_{0,\alpha }^{0,\beta }= -  \int_{t_1}^{t_2} \int_{\R^3}  \int_{\R^3} \omega_{\beta}^\alpha(x,v)  \big( \nabla_x^\alpha \nabla_v^\beta g(t,x,v)  \big)^2 \nabla_x\phi(t,x+a(v)t)\cdot \big(\nabla_v \omega_{\beta}^\alpha(x,v)
\]
\be\label{may2eqn261}
 - t\nabla_v a(v)\cdot \nabla_x \omega_{\beta}^\alpha(x,v)\big) d x d v d t .  
\ee
In the above equality, we used the fact that   $(\nabla_v -t\nabla_v a(v)\cdot\nabla_x)f(t,x+a(v)t)$ vanishes for any differentiable function ``$f(t,x)$''. 

Note that  there is  a loss of ``$t$'' each time when the derivative ``$\nabla_v$''  hits $\nabla_x\phi(t,x+a(v)t)$, which looks very problematic. To get around this issue, actually we can reduce it further by utilizing the   the Poisson equation satisfied by $\phi$ in (\ref{vp}). Note that, 
\be\label{april29eqn1}
 \phi(t,x+ a({v}) t)  = -\int_{\R^3} \int_{\R^3} e^{i x\cdot \xi +i  t(a(v)-a(u))\cdot \xi } |\xi|^{-2} \widehat{g}(t,\xi, u) d u  d\xi .
\ee
Therefore, for any fixed index $\beta$, after doing integration by parts in ``$u$''     many times, the following equality holds, 
\[
\nabla_{v}^\beta \big(\phi(t,x+ a({v})t)\big)=\sum_{|\gamma|\leq |\beta|} t^{|\gamma|} C^1_{\beta;\gamma}(v)\nabla_x^\gamma \phi(t,x+a(v)t) =\sum_{|\gamma|\leq |\beta|}  C^1_{\beta;\gamma}(v)
\]
\[
\times  \int_{\R^3} \int_{\R^3} e^{i x\cdot \xi +i  t(a(v)-a(u))\cdot \xi } {- t^{|\gamma|} (i\xi)^\gamma}{|\xi|^{-2}}\widehat{g}(t,\xi, u) d u  d\xi =\sum_{|\gamma|\leq |\beta|}  C^1_{\beta;\gamma}(v) \int_{\R^3} \int_{\R^3} e^{i x\cdot \xi +i  t(a(v)-a(u))\cdot \xi } \]
\[ 
\times \nabla_u \cdot \big( \frac{-i\nabla_u a(u)\cdot \xi}{| \nabla_u a(u)\cdot \xi|^2} \circ\cdots\big( \nabla_u \cdot \big( \frac{-i \nabla_u a(u)\cdot \xi}{|\nabla_u a(u)\cdot \xi|^2}  \frac{-  (i\xi)^\gamma}{|\xi|^2}\widehat{g}(t,\xi, u)\big)\cdots\circ \big) d u  d\xi=\sum_{|\gamma|\leq |\beta|}  C_{\beta;\gamma}(v) 
\]
\be\label{april29eqn2}
 \times \int_{\R^3} \int_{\R^3} e^{i x\cdot \xi +i   t(a(v)-a(u))\cdot \xi }  {m_\gamma(\xi)}{|\xi|^{-2}} c_{\gamma}(u) \nabla_u^\gamma \widehat{g}(t,\xi, u) d u  d \xi:= \sum_{|\gamma|\leq |\beta|}  C_{\beta;\gamma}(v) T_{\gamma}( \nabla_u^\gamma g )(t,x+a({v})t),
\ee
where $C^1_{\beta;\gamma}(v)$, $C_{\beta;\gamma}(v)$ and  $c_{\gamma}(u)$, $|\gamma|\leq |\beta|$, are  some uniquely determined coefficients and  $m_{\gamma}(\xi)$ are some uniquely determined $0$-order symbols, i.e., $m_{\gamma}(\xi)\in \mathcal{S}^\infty$. The following rough estimate holds for any $\gamma$ and  $\beta$, s.t., $|\gamma|\leq |\beta|\leq   N_0$,
\be\label{may4eqn55}
|C^1_{\beta;\gamma}(v)|+|C_{\beta;\gamma}(v)|\lesssim 1,\quad |c_{\gamma}(u)|\lesssim (1+|u|)^{|\beta|}, \quad \|m_{\gamma}(\xi)\|_{\mathcal{S}^\infty}\lesssim 1.  
\ee  	
 Therefore, from the equality (\ref{april29eqn2}), we can reduce the equality (\ref{april30eqn1}) as follows, 
 \[
   J_{\alpha_1,\alpha_2}^{\beta_1,\beta_2} =\sum_{|\gamma|\leq |\beta_1|}\int_{t_1}^{t_2} \int_{\R^3}  \int_{\R^3} C_{\beta_1;\gamma}(v)  \nabla_x^{\alpha_1+1} T_{\gamma}(\nabla_u^{ \gamma } g)(t,x+a(v) t )  \cdot  (\nabla_{x}^{\alpha_2} \nabla_v^{\beta_2+1} g(t,x ,v)   
 \]
\be\label{april30eqn3}
  - t C_{\beta_1, \beta_2}^{\alpha_1, \alpha_2}(v) \nabla_x^{\alpha_2+1}\nabla_v^{\beta_2}  g(t,x ,v))  \big( \omega_{\beta}^\alpha(x,v)\big)^2 \nabla_x^\alpha \nabla_v^\beta g(t,x,v)   d x d v d t.
\ee

\vspace{0.4\baselineskip}

\noindent\textbf{Step 2:}\quad The estimate of the case when there are more derivatives on the profile $g(t,x,v)$. 

\vspace{0.4\baselineskip}

We first consider the case when $|\alpha_2|+|\beta_2|\geq N_0-15$. Recall (\ref{may2eqn261}) and (\ref{april30eqn3}).  From the linear $L^\infty_{x}$-decay estimate (\ref{decayrvpphi}) in Lemma \ref{decayestimate} and the $L^2_{x,v}-L^\infty_{x,v}$ type bilinear estimate, we have
\be\label{april30eqn50}
\sum_{  N_0-15\leq |\alpha_2|+|\beta_2|\leq N_0 } |  J_{\alpha_1,\alpha_2}^{\beta_1,\beta_2}  |\lesssim \int_{t_1}^{t_2} (1+|t|)^{-1} \big( E_{\textup{high}}^{ }(t)\big)^2\big(E_{\textup{low}}^{ }(t) +|t|^{-1} E_{\textup{high}}^{ }(t)\big) d t , 
\ee

\vspace{0.4\baselineskip}

\noindent\textbf{Step 3:}\quad The case when  $|\alpha_2|+|\beta_2|<  N_0-15$. 

\vspace{0.4\baselineskip}
   For this case, we do dyadic decomposition for ``$T_{\gamma}(\nabla_u^{\gamma} g)(t,x+v t )$'' first. As a result, we have
\[
J_{\alpha_1,\alpha_2}^{\beta_1,\beta_2}= \sum_{k\in \mathbb{Z}}H_k, \quad H_k=\sum_{|\gamma|\leq |\beta_1|} \int_{t_1}^{t_2} \int_{\R^3}  \int_{\R^3} C_{\beta_1;\gamma}(v)  \nabla_x^{\alpha_1+1} T_{\gamma}(P_{k} [\nabla_u^{\gamma} g])(t,x+a(v) t )  \cdot  (\nabla_{x}^{\alpha_2} \nabla_v^{\beta_2+1} g(t,x ,v)
\]
\[
  - t  C_{\beta_1, \beta_2}^{\alpha_1, \alpha_2}(v)\nabla_x^{\alpha_2+1}\nabla_v^{\beta_2} g(t,x ,v))  \big( \omega_{\beta}^\alpha(x,v)\big)^2 \nabla_x^\alpha \nabla_v^\beta g(t,x,v)  d x d v.
\]
Recall (\ref{april29eqn2}). After first separating out the zero frequency of ``$\nabla_u^\beta \widehat{g}(t, \xi, u)$'' and then separating out the correction term $\nabla_u\cdot g_\beta(t,v)$ from the zero frequency,   the following decomposition holds,
\[
T_{\gamma}( \nabla_u^\gamma g )(t,x+vt)=  \int_{\R^3} \int_{\R^3} e^{i x\cdot \xi +i  t(a(v)-a(u))\cdot \xi }  {m_\gamma(\xi)} c_{\gamma}(u) {|\xi|^{-2}}   \big(\nabla_u^\gamma \widehat{g}(t,0, u) - \nabla_u\cdot g_{\gamma}(t,u)\big) c_{\beta}(u)d u  d \xi
\]
\[
-  \int_{\R^3} \int_{\R^3}   {m_\gamma(\xi)}{|\xi|^{-2}}  \nabla_u\big(c_{\gamma}(u) e^{i x\cdot \xi +i  t(a(v)-a(u))\cdot \xi }\big)  \cdot g_{\gamma}(t,u)     d u  d \xi
\]
\be\label{may2eqn61}
+ \int_0^1 \int_{\R^3} \int_{\R^3} e^{i x\cdot \xi +i  t(a(v)-a(u))\cdot \xi }   {m_\gamma(\xi)}{|\xi|^{-2}} c_{\gamma}(u) \xi\cdot \nabla_\xi \nabla_u^\beta \widehat{g}(t,s\xi , u) d u  d \xi d s.
\ee

From the decompositions (\ref{may2eqn61}), it is easy to see that the following estimate holds after using the volume of support of ``$\xi$'' if $2^{k}\leq |t|^{-1}$, 
\be\label{may2eqn101}
  \| \nabla_x^{\alpha_1+1} T_{\gamma}(P_{k} [\nabla_u^{ \gamma } g])(t,x  )\|_{L^2_x} \lesssim \sum_{|\alpha|\leq N_0} 2^{k/2}\big[ E_{\textup{low}} (t)   +\big(1+|t|2^{ k }\big) \|\widetilde{\omega}_\alpha(v)g_{\alpha}(t,v) \|_{L^2_v} +  2^{ k } E_{\textup{high}} (t)]. 
\ee
 From   the $L^2_{x}L^\infty_v-L^\infty_x L^2_v$ type bilinear estimate, the   above estimate (\ref{may2eqn101}), and the $L^\infty_xL^2_v$ decay estimate (\ref{decayrvpphi}) in Lemma \ref{decayestimate} and the Sobolev embedding, the following estimate holds, 
\be\label{april30eqn9}
  |H_k| \lesssim  \int_{t_1}^{t_2}     |t|^{-1/2} 2^{k/2}\big[\big( E_{\textup{low}} (t)+ 2^{ k } E_{\textup{high}} (t)\big)E_{\textup{high}}  (t)   +\big(1+|t|2^{ k }\big) \|\widetilde{\omega}_\alpha(v)g_{\alpha}(t,v) \|_{L^2_v}E_{\textup{high}}^2 (t)  ]  E_{\textup{high}}  (t)  d t.
\ee
If $2^{k}\geq |t|^{-1}$,  from the bilinear estimates (\ref{bilineardensity}) in Lemma \ref{bilineardensitylemma} and  the bilinear estimate (\ref{bilineardensitylargek}) in Lemma \ref{bilinearestimatelemma2}, and the estimate (\ref{may4eqn55}), the following estimate holds, 
\[
|H_k| \lesssim \int_{t_1}^{t_2}   |t|^{-1}   E_{\textup{high}} (t)  \big[\big(|t|^{-1}2^{-k }   E_{\textup{low}}^{ }(t)+ |t|^{-1/2} 2^{-|k|/2}E_{\textup{high}}^{ }(t)\big) E_{\textup{high}} (t)
\]
\be\label{april30eqn10}
 +\mathbf{1}_{(-\infty,0]}(k) \| \widetilde{\omega_\alpha}(v) g_{\alpha}(t,v)\|_{L^2_v} E_{\textup{high}}^2 (t)   \big] d t.
\ee
To sum up, from the estimates  (\ref{april30eqn9}) and (\ref{april30eqn10}), we have
\[
\sum_{   |\alpha_2|+|\beta_2|< N_0-15} | J_{\alpha_1,\alpha_2}^{\beta_1,\beta_2}  |\lesssim \sum_{|\alpha|\leq N_0} \int_{t_1}^{t_2}  \sum_{k\in \mathbb{Z}, 2^{k}\leq   |t|^{-1}}      |t|^{-1/2} 2^{k/2}\big[\big( E_{\textup{low}} (t)+ 2^{ k } E_{\textup{high}} (t)\big)E_{\textup{high}}  (t)   +\big(1+|t|2^{ k }\big)\]
\[
\times  \|\widetilde{\omega}_\alpha(v)g_{\alpha}(t,v) \|_{L^2_v}E_{\textup{high}}^2 (t)  ]  E_{\textup{high}}  (t)  d t +\int_{t_1}^{t_2}  \sum_{k\in \mathbb{Z}, 2^{k}\geq   |t|^{-1}}  |t|^{-1}   E_{\textup{high}} (t)  \big[  \big(|t|^{-1}2^{-k }   E_{\textup{low}}^{ }(t)+ |t|^{-1/2} 2^{-|k|/2} \]
\[
\times E_{\textup{high}}^{ }(t)\big) E_{\textup{high}} (t) + \mathbf{1}_{(-\infty,0]}(k) \| \widetilde{\omega_\alpha}(v) g_{\alpha}(t,v)\|_{L^2_v} E_{\textup{high}}^2 (t)   \big] d t  \lesssim \sum_{|\alpha|=N_0} \int_{t_1}^{t_2}    |t|^{-1}
\]
\vspace{-0.5\baselineskip}
\be\label{april30eqn51}
\times  E_{\textup{high}}^{ } (t) \big[ E_{\textup{low}}^{ } (t)E_{\textup{high}}^{ } (t) + |t|^{-1/2} \big( E_{\textup{high}}^{ } (t)\big)^2 +\log(1+|t|) \| \widetilde{\omega_\alpha}(v) g_{\alpha}(t,v)\|_{L^2_v} E_{\textup{high}}^2 (t)  \big] d t.
\ee
 Our desired estimate (\ref{highordervp}) follows from combining the estimates (\ref{april30eqn50}) and (\ref{april30eqn51}),.

Recall (\ref{may5eqn1}). Note that $g_{\alpha}(t,v)=0$ when $|\alpha|< N_0$. Our  desired estimate (\ref{highordervp2}) holds from rerun the the above argument  with minor modifications. We omit details here.

Lastly, recall (\ref{may5eqn1}) and (\ref{vp}). It is easy to see that the  desired estimate (\ref{may7eqn31}) holds straightforwardly  from the decay estimate (\ref{decayrvpphi}) in Lemma \ref{decayestimate}. 
 
 \end{proof}

\subsection{Proof of Proposition \ref{loworderenergy}}\label{lowprop} 

\begin{proof}
Recall the definition of the low order energy in (\ref{loworderenergydef}) and the definition of the  correction term $g_\alpha(t,v)$ in (\ref{may5eqn1}). From the equation satisfied by the profile $g(t,x,v)$ in (\ref{profilevp}), we have
\[
\sum_{|\alpha|\leq N_0}\p_t\big(\nabla_v^\alpha \widehat{g}(t,0,v)- \nabla_v \cdot g_\alpha(t,v)\big)
\]
\[
= \int_{\R^3}  \sum_{|\alpha|\leq N_0}  \nabla_v^\alpha\big(  \nabla_x \phi(t,x+a(v)t)\cdot\big(-t\nabla_v a (v)\cdot \nabla_x g(t,x,v) +\nabla_v g(t,x,v) \big)\big) d x 
 \]
\be\label{may5eqn6}
 - \int_{\R^3} \sum_{|\alpha|=N_0} \nabla_v\cdot\big( \nabla_x\phi( t,x+a(v)t) \nabla_v^\alpha g(t,x,v)   \big)  d x = \sum_{|\alpha|\leq N_0,  |\beta|+|\gamma|\leq |\alpha|, |\gamma|\leq N_0-1} H_{\beta, \gamma}^\alpha     (t,v),\ee
where
\be\label{may5eqn2}
H_{\beta, \gamma}^\alpha (t,v)= \int_{\R^3} \nabla_x \nabla_{v}^\beta\big(\phi(t,x+a(v)t)\big)\cdot\big(\nabla_v^{\gamma+1} g(t,x,v)-t C^\alpha_{\beta;\gamma}(v)\cdot \nabla_x\nabla_v^{\gamma} g(t,x,v) \big)d x, 
\ee
 where  $C_{\beta;\gamma}^\alpha (v)$, $|\beta|+|\gamma|\leq N_0$, $|\gamma|\leq N_0-1$, are some uniquely determined coefficients that satisfy the following rough estimate,
 \be\label{may5eqn22}
 |C_{\beta;\gamma}^\alpha(v)|\lesssim 1.
 \ee Recall the equality (\ref{april29eqn2}), after replacing $\nabla_v^\beta\big(\phi(t,x+a(v)t)\big) $ by  $\sum_{|\iota|\leq |\beta|}  C_{\beta;\iota}(v) T_{\iota}( \nabla_u^\iota g )(t,x+a({v})t) $ and  doing dyadic decomposition for $T_{\iota}(\nabla_u^\iota g)$, we have, 
 \[
 H_{\beta, \gamma}^\alpha (t,v)=\sum_{k\in \mathbb{Z}}H_{\beta, \gamma}^{
 \alpha; k}(t,v),\quad H_{\beta, \gamma}^{\alpha; k}(t,v)= \sum_{|\iota|\leq |\beta|} \int_{\R^3} C_{\beta;\iota}(v) \nabla_x T_{\iota}(P_k[\nabla_u^\iota g]) (t,x+a(v)t) \cdot\big(\nabla_v^{\gamma+1} g(t,x,v)\]
 \be\label{may5eqn11}
 -t C^\alpha_{\beta;\gamma}(v)\cdot \nabla_x\nabla_v^{\gamma} g(t,x,v) \big)d x. 
 \ee

 Based on the possible size of $k$, we separate into two cases as follow. 

\noindent $\bullet$ \quad The case when $k\geq 0.$ \qquad Recall (\ref{april29eqn2}). For the case when $|\gamma|\geq N_0-15$, we use the $L^\infty_{x }-L^1_xL^2_v$ type bilinear estimate by  putting $\nabla_x T_{\iota}(P_k[\nabla_u^\iota g])(t,x )$ in $L^\infty_{x }$.  For the case when $|\gamma|\leq N_0-15$, we use the  bilinear estimate (\ref{bilineardensitylargek}) in Lemma \ref{bilineardensitylemma}. As a result, the following estimate holds from  the estimate of coefficients in (\ref{may4eqn55}) and (\ref{may5eqn22}) and  the decay estimate (\ref{decayrvpphi}) in Lemma \ref{decayestimate},
\be\label{may4eqn1}
\sum_{|\alpha|\leq N_0,  |\beta|+|\gamma|\leq |\alpha|, |\gamma|\leq N_0-1}  \| \widetilde{\omega}_{\alpha}(v) H_{\beta, \gamma}^{\alpha; k}(t,v)\|_{L^2_v}\lesssim 2^{-k/2}  (1+|t|)^{-2} \big( E_{\textup{high}}^{ }(t)\big)^2  .
\ee
\noindent $\bullet$ \quad The case when $k\leq 0.$ \qquad For this case, we first do integration by parts in ``$x$'' to move around the ``$\nabla_x$'' derivative in front of ``$\nabla_x\nabla_v^{\gamma}   g(t,x,v)$''.  As a result, we have
\[
H_{\beta, \gamma}^{\alpha;k}(t,v)= \sum_{|\iota|\leq |\beta|} C_{\beta;\iota}(v)\big[  \int_{\R^3} \nabla_x T_{\iota}(P_k[\nabla_u^\iota g]) (t,x+a(v)t) \cdot \nabla_v^{\gamma+1} g(t,x,v)
 \]
 \be\label{may4eqn20}
 + t  \big(C^{\alpha}_{\beta;\gamma}(v)\cdot \nabla_x\big)\cdot\nabla_x T_{\iota}(P_k[\nabla_u^\iota g])(t,x+ a(v)t )    \nabla_v^{\gamma}    g(t,x,v) d x\big]. 
  \ee
 We use the same strategy used in the case    $k\geq 0$ to handle the above integral. More precisely,  we use the  $L^\infty_{x }-L^1_xL^2_v$  type bilinear estimate if $|\gamma|\geq N_0-15$  and use  the  bilinear estimate (\ref{bilineardensitylargek}) in Lemma \ref{bilineardensitylemma} if $|\gamma|\leq N_0-15$. As a result, the following estimate holds from the estimate of coefficients in (\ref{may4eqn55}) and (\ref{may5eqn22}), the decay estimate (\ref{decayrvpphi}) in Lemma \ref{decayestimate}, and  the  bilinear estimate (\ref{bilineardensitylargek}) in Lemma \ref{bilineardensitylemma},
\[
\sum_{|\alpha|\leq N_0,  |\beta|+|\gamma|\leq |\alpha|, |\gamma|\leq N_0-1} \| \widetilde{\omega}_{\alpha}(v) H_{\beta, \gamma}^{\alpha; k}(t,v)\|_{L^2_v}\lesssim   \big[ \min\{ (1+|t|)^{-2} + (1+|t|)^{-3}2^{-k}, (1+|t|)2^{3k} +2^{2k}\}
\]
 \be\label{may4eqn41}
 +\min\{ (1+|t|)^{-2} , (1+|t|)^{-1}2^{k} \} \big]\big( E_{\textup{high}}^{ }(t)\big)^2  .
 \ee

To sum up,   our desired estimate (\ref{may4eqn45}) holds from the decomposition (\ref{may5eqn6}) and  the estimates (\ref{may4eqn1}) and (\ref{may4eqn41}). We remark that the loss of ``$(1+t)^{\delta}$'' is caused by the summation with respect to  $k\in \mathbb{Z}$ when    $(1+|t|)^{-1}\leq 2^k\leq 1$. 
\end{proof}

\end{document}